\documentclass[a4paper]{article}

\usepackage[english]{babel}
\usepackage[utf8x]{inputenc}
\usepackage[T1]{fontenc}

\usepackage{amsmath}
\usepackage{graphicx}
\usepackage[colorinlistoftodos]{todonotes}
\usepackage{amsthm}
\usepackage{amsmath}
\usepackage{amssymb}
\usepackage{amsfonts}
\usepackage{multirow}
\usepackage{enumerate}
\usepackage{enumitem}
\usepackage{hyperref}

\newtheorem{con}{Conjecture}

\newtheorem{thm}{Theorem}
\newtheorem*{thm*}{Theorem}
\newtheorem{cor}[thm]{Corollary}

\theoremstyle{definition}

\newtheorem{clm}[thm]{Claim}
\newtheorem*{clm*}{Claim}
\newtheorem{obs}[thm]{Observation}
\newtheorem*{obs*}{Observation}

\newcommand{\abs}[1]{\left\vert#1\right\vert}
\DeclareMathOperator{\ex}{ex}

\allowdisplaybreaks

\begin{document}

\title{On the maximum number of odd cycles in graphs without smaller odd cycles\thanks{This work was initiated when the first author was a postdoctoral researcher at the Department of Computer Science and DIMAP, University of Warwick, Coventry CV4 7AL, UK, supported by the Engineering and Physical Sciences Research Council Standard Grant number EP/M025365/1. The visit of the second author at the University of Warwick was supported by the Leverhulme Trust 2014 Philip Leverhulme Prize of Daniel Kr\'al'.}}

\author{Andrzej Grzesik\thanks{Faculty of Mathematics and Computer Science, Jagiellonian University, {\L}ojasiewicza 6, 30-348 Krak\'{o}w, Poland. E-mail: {\tt Andrzej.Grzesik@uj.edu.pl}. The work of this author was supported by the National Science Centre grant number 2016/21/D/ST1/00998.}\and
Bartłomiej Kielak\thanks{Faculty of Mathematics and Computer Science, Jagiellonian University, {\L}ojasiewicza 6, 30-348 Krak\'{o}w, Poland and Institute of Mathematics of the Polish Academy of Sciences, \'Sniadeckich 8, 00-656 Warszawa, Poland. E-mail: \texttt{bartlomiej.kielak@doctoral.uj.edu.pl}.}}

\date{}

\maketitle

\begin{abstract}
We prove that for each odd integer $k \geq 7$, every graph on $n$ vertices without odd cycles of length less than $k$ contains at most $(n/k)^k$ cycles of length $k$.
This extends the previous results on the maximum number of pentagons in triangle-free graphs, conjectured by Erd\H os in 1984, and asymptotically determines the generalized Tur\'an number $\mathrm{ex}(n,C_k,C_{k-2})$ for odd $k$. In contrary to the previous results on the pentagon case, our proof is not computer-assisted.
\end{abstract}

\section{Introduction}

In 1984, Erd\H{o}s \cite{Erdos} conjectured that every triangle-free graph on $n$ vertices contains at most $(n/5)^5$ cycles of length 5 and the maximum is attained at the balanced blow-up of a~$C_5$. Gy\H{o}ri \cite{Gyori} proved an upper bound within a~factor 1.03 of the optimal. Using flag algebras method, Grzesik \cite{Grzesik} and, independently, Hatami, Hladk\'{y}, Kr\'{a}l', Norine, and Razborov \cite{HHKNR} proved that any triangle-free graph on $n$ vertices has at most $(n/5)^5$ copies of $C_5$, which is a~tight bound for $n$ divisible by $5$. Michael \cite{Michael} presented a sporadic counterexample to the characterization of the extremal cases by presenting a graph on $8$ vertices showing that not only a balanced blow-up of a $C_5$ can achieve the maximum. Recently, Lidick\'{y} and Pfender \cite{LP}, also using flag algebras, completely determined the extremal graphs for every $n$ by showing that the graph pointed out by Michael is the only extremal graph which is not a balanced blow-up of a pentagon.

Here, we prove the extension of the above results by showing the following theorem.

\begin{thm}\label{thm:main}
For each odd integer $k\ge7$, any graph on $n$ vertices without odd cycles of length smaller than $k$ contains at most $(n/k)^k$ cycles of length~$k$. Moreover, the balanced blow-up of a $k$-cycle is the only graph attaining this maximum.
\end{thm}

It is worth mentioning that, in contrary to the previous results on the pentagon case, our proof is not using flag algebras and is not computer-assisted. 

Estimating the maximum number of edges in an $H$-free graph on $n$ vertices, called the Tur\'an number of $H$ and denoted by $\ex(n,H)$, is one of the most well-studied problems in graph theory. The original Tur\'an Theorem \cite{Turan} solves it for cliques and the classical Erd\H os–Stone–Simonovits Theorem \cite{ESS} determines the asymptotic behavior of $\ex(n,H)$ for any other non-bipartite graph~$H$. The remaining bipartite case contains many interesting and longstanding open problems, as well as important results, see for example surveys by F\"uredi and Simonovits \cite{FS}, Sidorenko \cite{Sidorenko} or, in the case of cycles, the survey by Verstra\"ete~\cite{Verstraete}.

Generalization of the Tur\'an number, calculating the maximum possible number of copies of a graph $T$ in any $H$-free graph on $n$ vertices, denoted by $\ex(n,T,H)$, is attracting recently a lot of attention. Some specific cases, including the above mentioned case of $\ex(n,C_5,C_3)$, were considered earlier, but systematic studies of this problem were initiated by Alon and Shikhelman \cite{AS}. Especially in the case of cycles many results lately appeared. In particular, Bollob\'as and Gy\H ori \cite{BG} proved that $\ex(n,C_3,C_5) = \Theta(n^{3/2})$, Gy\H ori and Li \cite{GL} extended this result to obtain bounds for $\ex(n,C_3,C_{2k+1})$, which were later improved by Alon and Shikhelman \cite{AS} and by F\"uredi and \"Ozkahya \cite{FO}. Recently, Gishboliner and Shapira \cite{GS} proved a correct order of magnitude of $\ex(n,C_k, C_\ell)$ for each $k$ and $\ell$ and independently Gerbner, Gy\H ori, Methuku, and Vizer \cite{GGMV} for all even cycles, together with the tight asymptotic value of $\ex(n,C_4, C_{2k})$.

Theorem \ref{thm:main} implies the tight asymptotic value of $\ex(n,C_{k}, C_{k-2})$ for all odd~$k$, unknown before.

\begin{cor}
For any odd integer $k\ge7$, $\ex(n,C_{k}, C_{k-2}) = (n/k)^k + o(n^k)$. 
\end{cor}

The proof of the~corollary is a~standard application of the Graph Removal Lemma. If a~large graph $G$ is $C_{k-2}$-free, then, by the Regularity Lemma, it has at most $o(n^{\ell})$ copies of $C_\ell$ for any odd~$\ell$ smaller than $k$. By the Graph Removal Lemma, we can eliminate all the copies by removing $o(n^2)$ edges, thus the number of copies of $C_k$ in a~graph $G$ would change by at most $o(n^k)$.

The considered problem is closely related to the problem of finding the maximum number of induced cycles of a given length. Pippenger and Golumbic \cite{PG} conjectured in 1975 that for each $k\ge5$, any graph on $n$ vertices contains at most $n^k/(k^k-k)$ induced $k$-cycles and the extremal graphs are iterated blow-ups of~$C_k$. This conjecture was confirmed by Balogh, Hu, Lidick\'{y}, and Pfender~\cite{BHLP} for $k=5$. In their original paper, Pippinger and Golumbic proved a general bound for each $k\ge5$ within a multiplicative factor of $2e$. This was recently improved to $128e/81$ by Hefetz and Tyomkyn \cite{HT} and to $2$ by Kr\'{a}l', Norin, and Volec \cite{KNV}. Our main result is based on the method they developed.

\section{Main result}

Fix an odd integer $k \geq 7$ and let $G$ be any graph without $C_\ell$ for all odd $\ell$ between $3$ and $k-2$. Since there are no odd cycles smaller than $k$, each $k$-cycle in $G$ is induced. 

We bound the number of $k$-cycles by bounding the probability that sampling vertices of $G$ one by one at random results in a fixed induced $k$-cycle. However, instead of sampling the vertices in the cycle order, we do it with a small shift and sample the forth vertex before the third. This is to avoid the situation that a particular $3$-vertex induced path in $G$ cannot be extended to a $k$-cycle, which happens, for example, when $G$ is a blow-up of a $k$-cycle. 

For any $k$-cycle $v_0v_1\ldots v_{k-1}$ contained in $G$, by a \emph{good sequence} we denote a sequence $D = (z_i)_{i=0}^{k-1}$, where $z_i = v_i$ for $i \le 1$ and $i \geq 4$, $z_2 = v_3$, and $z_3 = v_2$, i.e., $v_2$ and $v_3$ are in the reversed order. 
Note that there are $2k$ different good sequences corresponding to a single induced $k$-cycle.
For any vertices $v$ and $w$, by $d(v,w)$ we denote the minimum distance between the vertices $v$ and $w$ in $G$. Also, for $v \in V(G)$, write $N(v)$ for the neighborhood of a~vertex $v$.

For a~fixed good sequence $D$, we define the following sets:
\begin{align*}
A_0(D) & = V(G),\\
A_1(D) & = N(z_0),\\
A_2(D) & = \{ w \not\in N(z_0) : d(z_1,w) = 2\},\\
A_3(D) & = N(z_1) \cap N(z_2),\\
A_4(D) & = \{w : z_0z_1z_3z_2w \textrm{ is an induced path}\},\\
A_i(D) & = \{w : z_0z_1z_3z_2z_4 \ldots z_{i-1}w \textrm{ is an induced path}\} \textrm{ for } 5 \leq i \leq k-2,\\
A_{k-1}(D) & = \{w : z_0z_1z_3z_2z_4 \ldots z_{k-2}w \textrm{ is an induced cycle}\}.
\end{align*}

We then define a~\emph{weight} $w(D)$ of a~good sequence $D$ as
\begin{align*}
w(D) = \prod_{i=0}^{k-1} \abs{A_i(D)}^{-1} = \frac 1n \prod_{i=1}^{k-1} \abs{A_i(D)}^{-1}.
\end{align*} 

This quantity has the following probabilistic interpretation: suppose we want to sample $k$ vertices $w_0$, \ldots, $w_{k-1}$, so that $(w_i)_{i=0}^{k-1}$ is a~good sequence. We start with choosing $w_0$ at random from all vertices of $G$. Next, we pick any neighbor of $w_0$ to be $w_1$. In general, $w_i$ is a~random vertex from the set $A_i((w_i)_{i=0}^{i-1})$ (note that the definition of $A_i(D)$ depends only on first $i$ elements of a~sequence~$D$). Then, $w(D)$ is just the probability that the sequence $(w_i)_{i=0}^{k-1}$ obtained in this random process is equal to $D$.

In particular, the sum of weights of all good sequences is at most one, since it is the sum of probabilities of pairwise disjoint events.

Fix a $k$-cycle $v_0v_1\ldots v_{k-1}$ in $G$, let $C=\{v_0,v_1,\ldots, v_{k-1}\}$ be the set of its vertices, and let $D_j = (v_j, v_{j+1}, v_{j+3}, v_{j+2}, v_{j+4}, \ldots, v_{j+k-1})$, for $0 \leq j \leq k-1$, where the indices are considered modulo $k$, be all the good sequences with the same orientation corresponding to this cycle (half of the total number of good sequences corresponding to this cycle). 

If we prove that 
$$\left(2\sum_{j=0}^{k-1} w(D_j)\right)^{-1} \leq M$$
for some number $M$, then $2\sum_{j=0}^{k-1} w(D_j) \geq M^{-1}$. Thus, by summing over all $k$-cycles (with both orientations) and using the fact that the sum of weights of all good sequences is at most one, we get that the total number of $k$-cycles is upper bounded by $M$.

Denote $n_{i,j} = \abs{A_i(D_j)}$. Since
$$\left(2\sum_{j=0}^{k-1} w(D_j)\right)^{-1} = \left(2\sum_{j=0}^{k-1} \prod_{i=0}^{k-1} n_{i,j}^{-1} \right)^{-1} = n\left(\sum_{j=0}^{k-1} \left(\frac{n_{1,j}}{2}\right)^{-1}\prod_{i=2}^{k-1} n_{i,j}^{-1} \right)^{-1},$$
the maximum possible value of
\begin{equation}\label{eq:maximizing_expression}
n\left(\sum_{j=0}^{k-1} \left(\frac{n_{1,j}}{2}\right)^{-1}\prod_{i=2}^{k-1} n_{i,j}^{-1} \right)^{-1}
\end{equation}
is an upper bound on the number of $k$-cycles in $G$.

Using the inequality between harmonic mean and geometric mean of $k$ terms and the inequality between geometric mean and arithmetic mean of $k(k-1)$ terms, we obtain
\begin{eqnarray*}
n\left(\sum_{j=0}^{k-1} \left(\frac{n_{1,j}}{2}\right)^{-1}\prod_{i=2}^{k-1} n_{i,j}^{-1} \right)^{-1} &\le&
\frac{n}{k}\left(\prod_{j=0}^{k-1} \frac{n_{1,j}}{2}\prod_{i=2}^{k-1} n_{i,j} \right)^{\frac{1}{k}}\\
&\le& \frac{n}{k}\left(\frac{1}{k(k-1)}\sum_{j=0}^{k-1} \left( \frac{n_{1,j}}{2}+\sum_{i=2}^{k-1} n_{i,j} \right)\right)^{k-1}.
\end{eqnarray*}

\begin{clm}\label{claim:sum_of_contributions}
The following inequality holds: 
$$\sum_{j=0}^{k-1} \left( \frac{n_{1,j}}{2} + \sum_{i=2}^{k-1} n_{i,j}  \right) \leq n(k-1),$$
with equality if and only if each vertex of $G$ is connected to two vertices of $C$ at distance two. 
\end{clm}
\begin{proof}
It is enough to prove that the contribution of any vertex $w \in V(G)$ to the above sum is at most $k-1$, and that such a~contribution can only occur if $w$ is connected to two vertices of $C$ at distance two. 

Notice that any vertex $w\in V(G)$ has at most 2 neighbors in $C$, since otherwise it creates a shorter odd cycle. For the same reason, each vertex $w$ satisfies the following property:

\begin{enumerate}[label=($\star$)]
\item There are at most three vertices in $C$ at distance exactly 2 from $w$, and any two such vertices are not adjacent.\label{no_neighbors_in_N2}
\end{enumerate}

If $w$ has no neighbors in $C$, then, for each $j$, it can contribute only to $n_{2,j}$. Moreover, if for some $j$ we have $d(w, v_{j}) = 2$, then $d(w, v_{j-1}) > 2$ and $d(w, v_{j+1}) > 2$ by \ref{no_neighbors_in_N2}, and so $w$ does not contribute to $n_{2,j}$ and $n_{2,j-2}$. Therefore, such $w$ contributes in total by at most $k-2$.

Assume, then, that $w$ has exactly one neighbor in $C$ --- from symmetry, let it be $v_0$. Because of having only one neighbor, for each $j$, $w$ does not contribute to $n_{3,j}$ and $n_{k-1,j}$. In order to contribute to $n_{i,j}$ for $i \notin \{2, 3, k-1\}$, $w$ needs to be connected to $v_{i+j-1}$, and so it can contribute only to $n_{1,0}$ and $n_{i,k-i+1}$ for $4\le i\le k-2$. Finally, $w$ can contribute to $n_{2,j}$ only if $d(w,v_{j+1}) = 2$ and $w \notin N(v_j)$. By \ref{no_neighbors_in_N2}, there are at most three vertices in $C$ at distance 2 from~$w$, but one of them is $v_1$ and $w \in N(v_0)$, so $w$ contributes to $\sum_{j=0}^{k-1}n_{2,j}$ by at most~2. It follows that in this case $w$ contributes to the considered sum in total by at most $k-3 + \frac 12$.

Finally, assume that $w$ has exactly two neighbors in $C$. These neighbors have to be at distance 2 in $C$, as otherwise it creates an odd cycle of length shorter than $k$. From symmetry, let $v_{k-1}$ and $v_1$ be the neighbors of $w$. Then, $d(w, v_i) = 2$ for $i = k-2$, $0$, $2$, and there are no more $i$ with this property by~\ref{no_neighbors_in_N2}. Therefore, $w$ contributes only to $n_{1,k-1}$, $n_{1,2}$, $n_{2,k-3}$, $n_{3,k-2}$, and $n_{i,k-i}$ for $4\le i\le k-1$, hence $w$ contributes to the considered sum in total by $k-1$.
\end{proof}

Using the above claim, we immediately get the wanted bound $(n/k)^k$ for \eqref{eq:maximizing_expression}. It follows that the total number of $k$-cycles in $G$ is at most $(n/k)^k$, as desired. 

If a graph $G$ is achieving this bound, then $n$ needs to be divisible by $k$ and we need to have equalities in all the inequalities we considered. In particular, for each $k$-cycle, all the other vertices of $G$ need to be connected with exactly two vertices of the cycle, which are at distance~2 (as in the blow-up of a $k$-cycle). Since we used the AM-GM inequality, all the blobs need to have the same size. Thus, one can easily deduce that the only graph attaining the maximum is the balanced blow-up of a $k$-cycle.

In the case of $n$ not divisible by $k$, in order to prove this way an exact bound on the number of $k$-cycles, one cannot use the AM-GM inequalities, but bound~\eqref{eq:maximizing_expression} using Claim~\ref{claim:sum_of_contributions} in a bit more sophisticated way. 
Trying just to maximize~\eqref{eq:maximizing_expression} over all such choices of $n_{i,j} \in \mathbb{N}$ that the inequality from Claim~\ref{claim:sum_of_contributions} holds, one can get higher value than for the numbers $n_{i,j}$ corresponding to blow-ups of a \hbox{$k$-cycle}, but such values may not be realizable by any graph. Therefore, using this approach, one would have to take into consideration also other relations between the numbers $n_{i,j}$. Still, if $n$ is big enough in relation to $k$, then the maximum of~\eqref{eq:maximizing_expression} needs to be achieved when there is an equality in Claim~\ref{claim:sum_of_contributions}. Thus, for $n$ big enough, the only graph achieving the maximum number of \hbox{$k$-cycles} is a~balanced blow-up of a $k$-cycle. 

\section{Concluding remarks and open problems}

In our proof, basically the only place where we are using that $k$ is an odd number is to say that if a $k$-cycle is not induced (or, more generally, there is a short path in the graph between distant vertices of this cycle), then the graph contains a smaller odd cycle. This is not the case if $k$ is an even number. Moreover, we do not have an analogue of Theorem \ref{thm:main} for even $k$, as forbidding any even cycle prevents from having big blow-ups of a~single edge. Nevertheless, one can carefully analyze the proof to obtain the following result on \emph{induced} even cycles.

\begin{obs}
For each even integer $k \geq 8$, any graph on $n$ vertices without induced cycles $C_\ell$ for $\ell = 3$ and $5 \leq \ell \leq k-1$ and without induced $C_6$ with one or two main diagonals contains at most $(n/k)^k$ induced cycles of length $k$.
\end{obs}

It seems possible that the same construction (balanced blow-up of a $k$-cycle) gives the best possible number of induced $k$-cycles also if we only forbid triangles. 

\begin{con}\label{con:induced_cycles}
For each integer $k\ge5$, any triangle-free graph on $n$ vertices contains at most $(n/k)^k$ induced cycles of length $k$.
\end{con}

Looking from the other side, if we forbid $C_\ell$ for some odd $\ell$ and try to maximize the number of $C_k$ for some larger odd $k$, it seems that, asymptotically, the best is always to take a balanced blow-up of an $(\ell+2)$-cycle. 

\begin{con}\label{con:odd_cycles}
For any odd integers $k>\ell\ge3$, it holds
$\ex(n,C_{k}, C_{\ell}) = \left(\binom{k}{\frac{k-(\ell+2)}{2}} + \binom{k}{\frac{k-3(\ell+2)}{2}} + \binom{k}{\frac{k-5(\ell+2)}{2}} + \ldots\right)\frac{n^k}{(\ell+2)^k} + o(n^k)$.
\end{con}

Using publicly available software Flagmatic \cite{flagmatic}, one can numerically verify that Conjecture \ref{con:induced_cycles} holds for $k\le8$ and Conjecture \ref{con:odd_cycles} holds for $k\le7$.

\end{document}